\documentclass[12pt]{amsart}
\usepackage{amssymb, amsthm, amsmath}
\newtheorem{question}{Question}[section]
\newtheorem{theorem}[question]{Theorem}
\newtheorem{lemma}[question]{Lemma}
\newtheorem{corollary}[question]{Corollary}

\title{Increasing chains and discrete reflection of cardinality}

\author{
Santi Spadaro}

\address{Institute of Mathematics, Silesian University in Opava, Na Rybn\' i\v cku 626/1, 746 01 Opava, Czech Republic}

\email{santispadaro@yahoo.com}

\thanks{Research partially supported by an INdAM-Cofund fellowship}

\subjclass[2000]{54A25}

\keywords{discrete set, free sequence, elementary submodel, strongly discretely Lindel\"of, Arhangel'skii Theorem}

\begin{document}
\baselineskip.525cm

\begin{abstract}
Combining ideas from two of our previous papers (\cite{S1} and \cite{S2}), we refine Arhangel'skii Theorem by proving a cardinal inequality of which this is a special case: any increasing union of strongly discretely Lindel\"of spaces with countable free sequences and countable pseudocharacter has cardinality at most continuum. We then give a partial positive answer to a problem of Alan Dow on reflection of cardinality by closures of discrete sets.
\end{abstract}

\maketitle

\section{Introduction and notation}

All spaces are assumed to be Hausdorff. A set is discrete if each one of its points is isolated in the relative topology. While structurally very simple, discrete sets play an important role in Set-theoretic Topology. For example, by an old result of De Groot, the cardinality of every topological space where discrete sets are countable cannot exceed $2^\mathfrak{c}$, where $\mathfrak{c}$ denotes the cardinality of the continuum. 

If discrete sets have a strong influence on cardinal properties of topological spaces, their closure are often true \emph{mirrors} of global properties of a topological space (see \cite{ATW} and \cite{A}). A classical result of Arhangel'skii says that a topological space $X$ is compact if and only if the closures of its discrete sets are compact. Whether this remains true when compactness is replaced by the Lindel\"of property is a well-known open question of Arhangel'skii. Partial answers to this question have been provided in \cite{AB} and \cite{PT}.

Another well-studied open problem, also due to Arhangel'skii, is whether closures of discrete sets reflect cardinality in compact spaces. More precisely, Arhangel'skii asked whether $|\overline{D}|=|X|$ for every compact space $X$ and discrete set $D \subset X$. Dow provided consistent counterexamples to this question in \cite{D2}, while Efimov \cite{Ef} proved that compact dyadic spaces reflect cardinality. In answer to a question of Alan Dow, Juh\'asz and Szentmikl\'ossy \cite{JS} proved that under a slight weakening of the GCH, compact spaces of countable tightness also reflect cardinality.

Aurichi noted in \cite{A}, that if $X$ is an $L$-space, left separated in order type $\omega_1$, then $|\overline{D}| <|X|$, for every discrete set $D \subset X$, so, by Justin Moore's construction of a ZFC $L$-space, there are non-discretely reflexive Tychonoff spaces in ZFC. But as far as we know, the ZFC existence of a non-discretely reflexive compact space is still open.

Arhangel'skii's question continues to inspire attempts at partial positive solutions. In particular, the following question of Alan Dow is still open.

\begin{question}
(\cite{D2}) Is $g(X)=|X|$ for every compact separable space $X$?
\end{question}

Where $g(X)$ is defined as the supremum of the cardinalities of the closures of discrete sets in $X$.  We will provide a partial positive answer to the above question in the final part of our paper.

One of the most central results in the theory of cardinal invariants is Arhangel'skii's Theorem, which solved a 50 year old question of Alexandroff (see \cite{H} for a survey).

\begin{theorem}
Let $X$ be a Lindel\"of first-countable space. Then $|X| \leq \mathfrak{c}$.
\end{theorem}

Arhangel'skii's original proof of his theorem made use of a particularly strong kind of discrete set called \emph{free sequence}. A set $\{x_\alpha: \alpha < \kappa \}$ is called a \emph{free sequence} if for every $\beta < \kappa$ we have $\overline{\{x_\alpha: \alpha < \beta\}} \cap \overline{\{x_\alpha: \alpha \geq \beta \}}=\emptyset$. In \cite{S2} we showed how the supremum of the sizes of free sequences in the space $X$ ($F(X)$) could replace the tightness in a generalization of the Arhangel'skii Theorem due to Juh\'asz. With some additional help from the technique of elementary submodels, this resulted in a considerably shorter proof of Juh\'asz's Theorem.

\begin{theorem} \label{arhfree}
(\cite{S2}) Let $\{X_\alpha: \alpha < \lambda \}$ be an increasing chain of topological spaces such that $F(X_\alpha) \cdot L(X_\alpha) \cdot \psi(X_\alpha) \leq \kappa$, for every $\alpha < \lambda$. Then $|\bigcup_{\alpha < \lambda} X_\alpha| \leq 2^{\kappa}$.
\end{theorem}

Given a topological space $(X,\tau)$, $L(X)$ (the \emph{Lindel\"of number} of $X$) is the minimum cardinal $\kappa$ such that every cover of $X$ has a subcover of cardinality $\kappa$ and $\psi(X)$ (the \emph{pseudocharacter} of $X$) is defined as follows $\psi(X)=\sup \{\psi(x,X): x \in X \}$, and $\psi(x,X)=\min \{\kappa: (\exists \mathcal{U} \in [\tau]^\kappa)(\bigcap \mathcal{U}= \{x\}) \}$. 

The above theorem has been generalized by various authors, especially with the aim of improving it in a non-regular setting and to provide bounds for the cardinality of power-homogeneous spaces (see, for example, \cite{B1}, \cite{B2} and \cite{CCP} and \cite{CPR}). Here we present a new refinement in a completely different direction.  Putting together ideas from \cite{S1} and \cite{S2} we are able to replace the Lindel\"of number with its supremum on closures of free sequences $(FL(X))$ in Theorem $\ref{arhfree}$. As a byproduct we obtain that the cardinality of the union of an increasing chain of a strongly discretely Lindel\"of spaces of countable pseudocharacter with countable free sequences does not exceed the continuum. Although in \cite{AB}, Arhangel'skii and Buzyakova proved that $L(X) \leq F(X) \cdot FL(X)$ for every Tychonoff space $X$, their proof uses the Tychonoff separation axiom in an essential way (they consider a compactification of $X$), while we are only assuming $X$ to be Hausdorff. Notation and terminology follow \cite{E} for Topology and \cite{K} for Set Theory. 

Kunen's book \cite{K} contains a good introduction on elementary submodels submodel. Dow's article \cite{D} is the most comprehensive survey on applications of elementary submodels to Topology. Other good introductions to this last topic are \cite{FW} \cite{G}, \cite{H} and \cite{W}.

\section{Closures of discrete sets and increasing chains}

The proof of Theorem $\ref{basethm}$ does not present significant changes from that of the case $\lambda=1$ in Theorem $\ref{arhfree}$, as presented, for example, in \cite{Sd}. We nevertheless include it, for the reader's convenience. 

\begin{theorem} \label{basethm}
(Juh\'asz, essentially) Let $(X, \tau)$ be a space. Then $|X| \leq 2^{FL(X) \cdot \psi(X) \cdot F(X)}$
\end{theorem}

\begin{proof}
Let $FL(X) \cdot \psi(X) \cdot F(X)=\kappa$ and $M$ be a $\kappa$-closed elementary submodel of $H(\theta)$ where $\theta$ is a large enough regular cardinal, such that $X, \tau, \kappa \in M$, $\kappa \subset M$ and $|M|=2^\kappa$.

We claim that $X \subset M$. Suppose this is not the case and let $p \in X \setminus M$. For every $x \in X \cap M$ use the fact that $\psi(x,X) \leq \kappa$ to pick a $\kappa$-sized family $\mathcal{U}_x  \in M$ such that $\bigcap \mathcal{U}_x=\{x\}$. We actually have $\mathcal{U}_x \subset M$ (see, for example, Theorem 1.6 of \cite{D}), and we can use that to pick $U_x \in \mathcal{U}_x$ such that $x \in U_x$ and $p \notin U_x$.

Let $\mathcal{U}= \{U\in M \cap \tau: x \in U \wedge p \notin U\}$. Then $\mathcal{U}$ covers $X \cap M$. Suppose that for some $\beta < \kappa^+$ we have constructed points $\{x_\alpha: \alpha < \beta \} \subset X \cap M$ and subcollections $\{\mathcal{U}_\alpha: \alpha < \beta \}$ of $\mathcal{U}$ such that $|\mathcal{U}_\alpha| \leq \kappa$ for every $\alpha < \beta$ and $\overline{\{x_\alpha: \alpha < \gamma \}} \subset \bigcup \bigcup_{\alpha \leq \gamma} \mathcal{U}_\alpha$ for every $\gamma< \beta$.

Let $A \subset X$ be a $\kappa$-sized free sequence. Note that $|\overline{A}| \leq 2^\kappa$. Indeed, the set $RC(X)$ of all regular closed sets of $\overline{A}$ has cardinality at most $2^\kappa$. The closed pseudocharacter of a Hausdorff space is bounded by the product of the pseudocharacter and the Lindel\"of number, so $\psi_c(\overline{A}) \leq \kappa$. Now, for every $x \in \overline{A}$ choose a $\kappa$-sized family $\mathcal{U}_x \subset RC(X)$ such that $x \in Int(F)$ for every $F \in \mathcal{U}_x$ and $\bigcap \mathcal{U}_x=\{x\}$. The map $x \to \mathcal{U}_x$ is injective and hence $|\overline{A}| \leq (2^\kappa)^\kappa=\kappa$. From this observation it follows that if $A \in M$ and $|A| \leq \kappa$ then $\overline{A} \subset M$.

In particular, since $M$ is $\kappa$-closed we have that $\{x_\alpha: \alpha <\beta \} \in M$ and hence $\overline{\{x_\alpha: \alpha < \beta \}} \subset M$. Therefore, we can choose a $\kappa$ sized subcollection $\mathcal{U}_\beta$ of $\mathcal{U}$ covering $\overline{\{x_\alpha: \alpha < \beta \}}$. If $\bigcup_{\alpha \leq \beta} \mathcal{U}_\alpha$ does not cover $X \cap M$ pick a point $x_\beta \in X \cap M \setminus \bigcup_{\alpha \leq \beta} \mathcal{U}_\alpha$. If we didn't stop before reaching $\kappa^+$, then $\{x_\alpha: \alpha < \kappa^+ \}$ would be a free sequence of size $\kappa^+$ in $X$. Therefore, there is $\mathcal{V} \subset \mathcal{U}$ of size $\kappa$ such that $X \cap M \subset \bigcup \mathcal{V}$. Note that since $M$ is $\kappa$-closed we have $\mathcal{V} \in M$.Therefore $M \models X \subset \bigcup \mathcal{V}$ and hence $H(\theta) \models X \subset \bigcup \mathcal{V}$. So there is $V \in \mathcal{V}$ such that $p \in V$, which is a contradiction. 
\end{proof}

The proof of the increasing chain version of Theorem $\ref{basethm}$ relies on the following Lemmas. 

\begin{lemma} \label{shapfree}
Let $X$ be a space such that $FL(X) \leq \kappa$ and $\mathcal{U}$ be an open cover for $X$. Then there is a free sequence $F \subset X$ and a subcollection $\mathcal{V} \subset \mathcal{U}$ such that $|\mathcal{V}|=|F| \cdot \kappa$ and $X = \overline{F} \cup \bigcup \mathcal{V}$.
\end{lemma}

\begin{proof}
Suppose you have constructed, for some ordinal $\beta$, a free sequence $\{x_\alpha: \alpha < \beta \}$ and $\kappa$-sized subcollections $\{\mathcal{U}_\alpha: \alpha < \beta \}$ of $\mathcal{U}$ such that $\overline{\{x_\alpha : \alpha < \gamma\}} \subset \bigcup_{\alpha \leq \gamma} \bigcup \mathcal{U}_\alpha$ for every $\gamma < \beta$.

Let $\mathcal{U}_\beta$ be a $\kappa$-sized subcollection of $\mathcal{U}$ covering the subspace $\overline{\{x_\alpha : \alpha < \beta \}}$ and, if you can, pick a point $x_\beta \in X \setminus \bigcup_{\alpha \leq \beta} \bigcup \mathcal{U}_\beta$. Let $\mu$ be the least ordinal such that $$\overline{\{x_\alpha: \alpha < \mu\}} \cup \bigcup_{\alpha < \mu} \bigcup \mathcal{U}_\alpha=X.$$ Then $F=\{x_\alpha : \alpha < \mu\}$ is a free sequence and if we set $\mathcal{V}=\bigcup_{\alpha < \kappa} \bigcup \mathcal{U}_\alpha$ we have $|\mathcal{V}|=|F| \cdot \kappa$.
\end{proof}

\begin{lemma} \label{minusap}
For every $x \in X$ we have that $FL(X \setminus \{x\}) \leq FL(X) \cdot \psi(X)$.
\end{lemma}

\begin{proof}
Set $\kappa=FL(X) \cdot \psi(X)$ and let $F \subset X \setminus \{x\}$ be a free sequence in $X \setminus \{x\}$. Let $\mathcal{U}$ be a $\kappa$-sized family of open neighbourhood of $x$ such that $\bigcap \mathcal{U}=\{x\}$. Note that $F \subset \bigcup \{X \setminus U: U \in \mathcal{U} \}$, $F \setminus U$ is a free sequence in $X \setminus U$, and $FL(X \setminus U) \leq \kappa$ for every $U \in \mathcal{U}$. Now $Cl_{X \setminus \{x\}}(F)=\bigcup_{U \in \mathcal{U}} \overline{F} \setminus U$. Therefore $L(Cl_{X \setminus \{x\}}(F)) \leq \kappa$ and we are done.
\end{proof}

\begin{theorem} \label{firstmainthm}
Let $(X, \tau)$ be a topological space and $\{X_\alpha: \alpha < \lambda \}$ be an increasing chain of subspaces of $X$ such that $X=\bigcup_{\alpha < \lambda} X_\alpha$ and $FL(X_\alpha) \cdot F(X_\alpha) \cdot \psi(X_\alpha) \leq \kappa$. Then $|X| \leq 2^\kappa$.
\end{theorem}

\begin{proof}
If $\lambda \leq 2^\kappa$ then we are done by Theorem $\ref{basethm}$, so we can assume that $\lambda=(2^\kappa)^+$. 

Let $\mu$ be a large enough regular cardinal and choose an elementary submodel $M \prec H(\mu)$ such that $[M]^\kappa \subset M$, $|M|=2^\kappa$, and $\{X, \tau, \kappa, \lambda \} \cup \kappa \subset M$.

Call a set $C \subset X$ \emph{bounded} if $|C| \leq 2^\kappa$.

\vspace{.1in}

\noindent {\bf Claim 1.} If $C \in [X \cap M]^{\leq \kappa}$, then $\overline{C}$ is bounded.

\begin{proof}[Proof of Claim 1]
Claim 1 will be proved if we can show that $\overline{C} \subset X \cap M$. So, suppose that this is not true and choose $p \in \overline{C} \setminus M$. Choose $\theta$ large enough, so that $\overline{C} \cap M \subset X_\theta$. By $\psi(X_\theta) \leq \kappa$ we can find open neighbourhoods $\{U_\alpha: \alpha < \kappa \}$ of the point $p$ such that $X_\theta \setminus \{p\}=\bigcup_{\alpha < \kappa} X_\theta \setminus U_\alpha$. By Lemma $\ref{shapfree}$ we can find a free sequence $D_\alpha \subset X_\theta \setminus U_\alpha$ and relative open sets $\{V_{\alpha \beta}: \beta < \kappa \}$ in $X_\theta \setminus U_\alpha$ such that $X_\theta \setminus U_\alpha \subset \overline{D_\alpha} \cup \bigcup_{\beta < \kappa} V_{\alpha \beta}$ for every $\alpha < \kappa$. By $FL(X_\theta \setminus U_\alpha) \leq \kappa$ we can find relative open sets $\{G_{\alpha \beta}: \beta < \kappa \}$ in $X_\theta \setminus U_\alpha$ such that $\overline{D_\alpha} \subset \bigcup_{\beta < \kappa} G_{\alpha \beta}$, for every $\alpha < \kappa$.

Note that $p \notin \overline{V_{\alpha \beta}} \cup \overline{G_{\alpha \beta}}$, for every $\alpha, \beta < \kappa$. Setting $C_{\alpha \beta}=V_{\alpha \beta} \cap C$ and $E_{\alpha \beta}=G_{\alpha \beta} \cap C$ we then have:

$$\overline{C} \cap X_\theta \setminus \{p\} = \bigcup_{\alpha, \beta < \kappa} (\overline{C_{\alpha, \beta}} \cup \overline{E_{\alpha, \beta}}) \cap X_\theta$$

Note now that by $\kappa$-closedness of $M$, $C_{\alpha \beta} \in M$ and $E_{\alpha \beta} \in M$, for every $\alpha, \beta$ and $\theta$.

We have:

$$\overline{C} \cap M=\bigcup_{\alpha, \beta < \kappa} (\overline{C_{\alpha, \beta}} \cup \overline{E_{\alpha, \beta}}) \cap M$$

So:

$$M \models \overline{C}=\bigcup_{\alpha, \beta<\kappa} (\overline{C_{\alpha, \beta}} \cup \overline{E_{\alpha \beta}})$$

Which implies:

$$H(	\mu) \models \overline{C}=\bigcup_{\alpha, \beta<\kappa} (\overline{C_{\alpha \beta}} \cup \overline{E_{\alpha \beta}})$$

But that is a contradiction, because:

$$H(\mu) \models p \in \overline{C} \setminus \bigcup_{\alpha, \beta < \kappa} (\overline{C_{\alpha, \beta}} \cup \overline{E_{\alpha, \beta}})$$

\renewcommand{\qedsymbol}{$\triangle$}
\end{proof}

Now we claim that $X \subset M$. Suppose not and choose $p \in X \setminus M$.

\noindent {\bf Claim 2.} The collection $\mathcal{U}=\{U \in M \cap \tau: p \notin U \}$ is an open cover of $X \cap M$.

\begin{proof}[Proof of Claim 2]

Fix $x \in X \cap M$ and let $\mathcal{V}=\{V \in \tau: x \notin \overline{V} \}$. Note that $\mathcal{V} \in M$ and $\mathcal{V}$ covers $X \setminus \{x\}$. Suppose you have constructed subcollections $\{\mathcal{V}_\alpha: \alpha < \beta \} $ of $\mathcal{V}$ such that $\mathcal{V}_\alpha \in M$, $|\mathcal{V}_\alpha| \leq \kappa$ for every $\alpha < \beta$ and a free sequence $\{x_\alpha: \alpha < \beta \} \subset X \cap M$ such that $Cl_{X \setminus \{x\}}(\{x_\alpha: \alpha < \gamma\}) \subset \bigcup_{\alpha < \gamma} \mathcal{V}_\alpha$ for every $\gamma < \beta$. The set $Cl_{X \setminus \{x\}}(\{x_\alpha: \alpha < \beta \})$ is bounded, so we can find an ordinal $\lambda_\beta < \lambda$ such that $Cl_{X \setminus \{x\}}(\{x_\alpha: \alpha <\beta \}) \subset X_{\lambda_\beta}$. Since $FL(X_{\lambda_\beta}) \cdot \psi(X) \leq \kappa$, by Lemma $\ref{minusap}$ we have that the Lindel\"of number of $Cl_{X \setminus \{x\}}(\{x_\alpha: \alpha < \beta\})$ is at most $\kappa$ and hence we can pick a family $\mathcal{V}_\beta \in [\mathcal{V}]^{\leq \kappa}$ such that $Cl_{X \setminus \{x\}}(\{x_\alpha: \alpha < \beta\}) \subset \bigcup \mathcal{V}_\beta$. If $\bigcup_{\alpha \leq \beta} \mathcal{V}_\beta$ covers $X\setminus \{x\}$ we stop, otherwise we pick $x_\beta \in (X \setminus \{x\} \cap M) \setminus \bigcup_{\alpha \leq \beta} \mathcal{V}_\beta$. If we carried this on for $\kappa^+$ many steps, then $F=\{x_\alpha: \alpha < \kappa^+\}$ would be a free sequence of cardinality $\kappa^+$ in $X \setminus \{x\}$. Since $F$ is bounded, we can choose $\theta< \lambda$ such that $F \subset X_\theta$. So $L(Cl_{X_\theta}(F)) \leq \kappa$. But $F$ cannot converge to $x$, because every set of cardinality $\kappa^+$ of a space of Lindel\"of number $\kappa$ has a complete accumulation point. Therefore there is an open neighbourhood $U$ of $x$ which misses $\kappa^+$ many points of $F$. Therefore $F \setminus U$ is a free sequence in $X$ of cardinality $\kappa^+$, but that contradicts $F(X_\theta) \leq \kappa$.

So there is a family $\mathcal{W} \in [\mathcal{U}]^{\leq \kappa}$ such that $X \setminus \{x\} \subset \bigcup \mathcal{W}$. By elementarity, we can assume that $\mathcal{W} \in M$ and hence $\mathcal{W} \subset M$. Let now $W \in \mathcal{W}$ be such that $p \in W$. Then the set $U:=X \setminus \overline{W} \in M$ is a neighbourhood of $x$ which misses $p$.
\renewcommand{\qedsymbol}{$\triangle$}
\end{proof}

Suppose that for some $\beta < \kappa^+$ we have constructed a free sequence $\{x_\alpha: \alpha < \beta \} \subset X \cap M$ and subcollections $\{\mathcal{U}_\alpha: \alpha < \beta \}$ of $\mathcal{U}$ such that $\mathcal{U}_\alpha \in M$, $|\mathcal{U}_\alpha| \leq \kappa$ and $\overline{\{x_\gamma: \gamma < \alpha \}} \subset \bigcup \bigcup_{\gamma <\alpha} \mathcal{U}_\alpha$, for every $\alpha <\beta$. Since $\overline{\{x_\alpha: \alpha <\beta \}}$ is bounded, we have that $L(\overline{\{x_\alpha: \alpha < \beta\}}) \leq \kappa$ and hence we can find a subcollection $\mathcal{U}_\beta$ of $\mathcal{U}$ of size $\kappa$ such that $\overline{\{x_\alpha: \alpha < \beta\}} \subset \bigcup \mathcal{U}_\beta$. If $\bigcup_{\alpha \leq \beta} \mathcal{U}_\alpha$ does not cover $X \cap M$ we can find a point $x_\beta \in X \cap M \setminus \bigcup_{\alpha \leq \beta} \mathcal{U}_\alpha$. If we didn't stop before reaching $\kappa^+$, then $\{x_\alpha : \alpha < \kappa^+\}$ would be a $\kappa^+$-sized free sequence in $X$. But this can't happen, because $\{x_\alpha: \alpha < \kappa^+\}$ is bounded. So there is a $\mathcal{V} \in [\mathcal{U}]^{\leq \kappa}$ such that $X \cap M \subset \bigcup \mathcal{V}$. But since $M$ is $\kappa$-closed we have that $\mathcal{V} \in M$ and hence $M \models X \subset \bigcup \mathcal{V}$. Therefore $H(\mu) \models X \subset \bigcup \mathcal{V}$, and hence there is $V \in \mathcal{V}$ such that $p \in V$, which is a contradiction.

\end{proof} 

As a corollary, we find a result related to discrete reflection of cardinality, which will be the main subject of the next section.

\begin{lemma} \label{lemchar}
\cite{S1} Let $\kappa$ be an infinite cardinal and $X$ be a space where $|\overline{D}| \leq \kappa$ for every discrete $D \subset X$. Then $\psi(X) \leq \kappa$.
\end{lemma}
 
 \begin{proof}
Let $x \in X$. Now let $\mathcal{V}=\{V \subset X: V$ is open and $x \notin \overline{V} \}$. Then $\mathcal{V}$ covers $X \setminus \{x\}$ and hence we can find a discrete $D \subset X \setminus \{x\}$ and a subcollection $\mathcal{U} \subset \mathcal{V}$ with $|\mathcal{U}|=|D|$ such that $X \setminus \{x\} \subset \bigcup \mathcal{U} \cup \overline{D}$. So $(\bigcap_{x \in \overline{D} \setminus \{x\}} X \setminus \{x\}) \cap (\bigcap_{U \in \mathcal{U}} X \setminus \overline{U})=\{x\}$, which implies that $\psi(x,X) \leq \kappa$.
 \end{proof}

\begin{corollary}
Let $\{X_\alpha: \alpha < \lambda \}$ be an increasing chain of spaces such that $|\overline{D}| \leq \kappa$ for every discrete set $D \subset X_\alpha$ and every $\alpha < \lambda$. Then $|\bigcup_{\alpha < \lambda} X_\alpha| \leq 2^\kappa$.
\end{corollary}

\section{A reflection theorem for hereditarily normal spaces}

In \cite{D}, Dow asked whether compact separable spaces reflect cardinality. It is at present unknown even the following special case. Suppose that in some compact space $X$, the closure of every discrete set has size bounded by the continuum. Is then $|X| \leq \mathfrak{c}$? We are going to prove that this is the case if $X$ is hereditarily normal. As a matter of fact, the only feature of compactness that we need is the fact that pseudocharacter equals character at every point, and separability can be relaxed to the ccc. 

A \emph{cellular family} is a family of pairwise disjoint non-empty open sets. The \emph{cellularity} of $X$ is defined as follows: $c(X)=\sup \{|\mathcal{U}|: \mathcal{U}$ is a cellular family in $X \}$. Recall that a $\pi$-base in a topological space $X$ is a set $\mathcal{P}$ of non-empty open sets such that for every open set $U \subset X$ there is $P \in \mathcal{P}$ with $P \subset U$. The \emph{$\pi$-weight} of $X$ ($\pi w(X)$) is defined as the minimum cardinality of a $\pi$-base for $X$.

 Given a cardinal $\mu$, the logarithm of $\mu$ is defined as follows $\log(\mu)=\min \{\kappa: 2^\kappa \geq \mu \}$. We need a well-known, often used and easily proven lemma of Shapirovskii.

\begin{lemma}
(Shapirovskii) Let $X$ be a space and $\mathcal{U}$ be a cover of $X$. Then there is a discrete set $D \subset X$ and a subcollection $\mathcal{V} \subset \mathcal{U}$ such that $|D|=|\mathcal{V}|$ and $X=\overline{D} \cup \bigcup \mathcal{V}$.
\end{lemma}

\begin{theorem}
Let $X$ be a hereditarily normal space such that $\psi(x,X)=\chi(x,X)$ for every point $x \in X$ and $|\overline{D}| \leq 2^{c(X)}$ for every discrete set $D \subset X$. Then $|X| \leq 2^{c(X)}$
\end{theorem}

\begin{proof}
Set $\kappa=\log{(2^{c(X)})^+}$. Let $M$ be a $<\kappa$-closed elementary submodel of $H(\theta)$, for large enough regular $\theta$ such that $|M|=2^{c(X)}$ and $M$ contains everything we need.

\noindent {\bf Claim 1.}
For every $x \in X \cap M$ we have $\chi(x,X) \leq 2^\kappa$.

\begin{proof}[Proof of Claim 1]
Fix $x \in X \cap M$. Subclaim: \emph{for every $p \in X \setminus M$ we can find an open $U \in M$ such that $x \in U$ and $p \notin U$}. If that were true, then we could find a family $\mathcal{S}$ of open neighbourhoods of $x$ such that $|\mathcal{S}| \leq 2^\kappa$ and $\bigcap \mathcal{S} \subset X \cap M$. Now $|X \cap M| \leq 2^\kappa$, so $x$ would have pseudocharacter $2^\kappa$ in $X$, and since pseudocharacter and character in $X$ we would be done. To prove the subclaim, let $\mathcal{U}$ be the set of all open sets $U \subset X$ such that $x \notin \overline{U}$. Then $\mathcal{U} \in M$ and $\mathcal{U}$ covers $X \setminus \{x\}$. By Shapirovskii's lemma we can find a subcollection $\mathcal{W} \subset \mathcal{U}$ and a discrete set $D \subset X \setminus \{x\}$ such that $\mathcal{W} \in M$, $D \in M$ $|\mathcal{W}|=|D| \leq 2^\kappa$ and $X \setminus \{x\} \subset \overline{D} \cup \bigcup \mathcal{W}$. Observe that $\overline{D} \in M$ and $|\overline{D}| \leq 2^\kappa$ and hence $\overline{D} \subset X \cap M$. Therefore $p \notin \overline{D}$ and hence there is $W \in \mathcal{W}$ such that $p \in W$. Now $\overline{W} \in M$ and $x \notin \overline{W}$ therefore $X \setminus \overline{W} \in M$ is a neighbourhood of $x$ which misses $p$.
\renewcommand{\qedsymbol}{$\triangle$}
\end{proof}

\noindent {\bf Claim 2.} The set $X \cap M$ is dense in $X$.

\begin{proof}[Proof of Claim 2]
Suppose that is not the case. Then there is an open set $V \subset X$ such that $\overline{V} \cap X \cap M =\emptyset$. Let now $x \in X \cap M$ and choose an open set $U_0 \in M$ such that $U_0 \cap \overline{V}=\emptyset$. Suppose we have constructed, for some $\beta \in \kappa^+$ a cellular family $\{U_\alpha: \alpha < \beta \} \subset M$ such that $U_\alpha \cap \overline{V}=\emptyset$ for every $\alpha < \beta$. Then $X \setminus \overline{\bigcup_{\alpha < \beta} U_\alpha} \in M$ and given $y \in X \setminus \overline{\bigcup_{\alpha < \beta} U_\alpha} \cap M$ we can find an open set neighbourhood $U_\beta$ of $y$ such that $U_\beta \cap \overline{V}=\emptyset$. Now replace $U_\beta$ with its intersection with $X \setminus \overline{\bigcup_{\alpha < \beta} U_\alpha}$, which is still in $M$ as the intersection of two elements of $M$. Eventually, $\{U_\alpha: \alpha \in \kappa^+\}$ would be a $\kappa^+$-sized cellular family in $X$, which is a contradiction.
\renewcommand{\qedsymbol}{$\triangle$}
\end{proof}

Putting together Claim 1 and Claim 2 we get that $\pi w(X) \leq 2^\kappa$.

 We now claim that $X \subset M$. Indeed, suppose that this is not the case and let $p \in X \setminus M$.

\noindent {\bf Claim 3.} For every $x \in X \cap M$, there is an open set $V \in M$ such that $x \in V$ and $p \notin V$. 

\begin{proof}[Proof of Claim 3] Fix $x \in X \cap M$ and let $\mathcal{U}=\{V \in M: x \notin \overline{V} \}$. The set $\mathcal{U}$ covers $X \setminus \{x\}$. Use Shapirovskii's Lemma to find a discrete set $D \subset X \cap M$ such that $X \setminus \{x\} \subset \overline{D} \cup \bigcup \{U_x: x \in D \}$. By Shapirovskii's bound for the number of regular open sets (see \cite{J} or \cite{Mo} or \cite{BS} for a \emph{game-theoretic} proof) we have that $\rho(X) \leq \pi w(X)^{c(X)} \leq (2^\kappa)^\kappa=2^\kappa$. Moreover, since by Jones Lemma $\rho(X) \geq 2^{|D|}$ in every hereditarily normal space $X$, we must have $|D| < \kappa$ and hence $D \in M$. Therefore $\overline{D} \in M$. From $|\overline{D}| \leq 2^{c(X)}$ we get that $\overline{D} \subset X \cap M$ and thus $p \notin \overline{D}$. This implies that there is $x \in D$ such that $p \in U_x$. By letting $V=X \setminus \overline{U_x}$ we get that $V$ is a neighbourhood of $x$ such that $V \in M$ and $p\notin V$.
\renewcommand{\qedsymbol}{$\triangle$}
\end{proof}

If we now let $\mathcal{V}=\{U \in M: p \notin U \}$, we see that $\mathcal{V}$ is an open cover of $X \cap M$. Using Shapirovskii's Lemma again, we obtain the existence of a discrete set $E \subset X \cap M$ such that $X \cap M \subset \overline{E} \cup \bigcup \{U_x: x \in E \}$, where $U_x \in \mathcal{V}$ and $x \in U_x$. By the same reasoning as in the proof of the Claim we have that $\overline{E} \subset X \cap M$. The closure property of $M$ implies that $\overline{E} \cup \bigcup \{U_x: x \in E\} \in M$  and hence $M \models X \subset \overline{E} \cup \bigcup \{U_x: x \in E \}$. By elementarity, we get that $H(\theta) \models X \subset \overline{E} \cup \bigcup \{U_x: x \in E \}$ and therefore there is $x \in E$ such that $p \in U_x$, but that contradicts the definition of $\mathcal{V}$.

Therefore $X \subset M$ and we are done.
\end{proof}

Recall the definition of the \emph{depth} of $X$: $g(X)=\{|\overline{D}|: D \subset X$ discrete$\}$.

\begin{corollary} \label{maincor}
Let $X$ be a compact hereditarily normal ccc space such that $g(X) \leq \mathfrak{c}$. Then $|X| \leq \mathfrak{c}$,
\end{corollary}

Note that there are consistent examples of compact hereditarily normal hereditarily separable spaces of cardinality $2^\mathfrak{c}$ (for example, Fedorchuk's compact $S$-space), and this shows that the condition about the depth is essential in Corollary $\ref{maincor}$.

\section{Acknowledgements}

We thank the referee for correcting an error in the proof of Theorem $\ref{firstmainthm}$, INdAM for financial support and the Fields Institute of the University of Toronto for hospitality.

\end{document}